%% file: MathematikaSub.tex
\documentclass[12pt, reqno]{amsart}
\usepackage{amssymb,latexsym,amsmath,amsthm,enumitem,mathrsfs,geometry, verbatim}
\usepackage{fullpage}
\usepackage{fancyhdr}
\usepackage{color}
\usepackage{stmaryrd}
\usepackage{mathrsfs}
\usepackage{lineno}
\usepackage{hyperref}
\usepackage{diagbox}
\usepackage{array}
\usepackage{tikz}
\usetikzlibrary{arrows}
\usetikzlibrary{positioning}
\usepackage{float}
\usepackage{mathtools}
\usepackage{breqn}
\input{format}
\newcommand{\bigO}{\operatorname{O}}
\newcommand{\Q}{\mathbb{Q}}

\newcommand{\F}{\mathbb{F}}
\newcommand{\Z}{\mathbb{Z}}
\newcommand{\frakp}{\mathfrak{p}}

\newcommand{\Gal}{\mathrm{Gal}}

\newtheorem{remark}{Remark}
\newtheorem{corollary}{Corollary}
\newtheorem{proposition}[thm]{Proposition}

\numberwithin{equation}{section}

\definecolor{pink}{rgb}{1,.2,.6}
\definecolor{orange}{rgb}{0.7,0.3,0}
\definecolor{blue}{rgb}{.2,.6,.75}
\definecolor{green}{rgb}{.4,.7,.4}
\definecolor{purple}{RGB}{127,0,255}

\begin{document}
\title{
Equidistribution of $\alpha p^{\theta}$ with a Chebotarev condition\\ and applications to extremal primes
}
\author[Malik]{Amita Malik}
\address{American Institute of Mathematics, San Jose, California 95112, USA}
\email{amita.malik@aimath.org}
\author[Prabhu]{Neha Prabhu}
\address{Indian Institute of Science Education and Research Pune, Dr. Homi Bhabha Road, Pune - 411008, Maharashtra, India }
\email{neha.prabhu@acads.iiserpune.ac.in}
\thanks{}
\begin{abstract}
    {We establish a joint distribution result concerning the fractional part of $\alpha p^\theta$ for $\theta \in (0,1), \  \alpha>0$, where $p$ is a prime satisfying a Chebotarev condition in a fixed finite Galois extension over $\Q$. As an application, for a fixed non-CM elliptic curve $E/\Q$, an asymptotic formula is given for the number of primes at the extremes of the Sato-Tate measure modulo a large prime $\ell$. These are precisely the primes $p$ for which the Frobenius trace $a_p(E)$ satisfies the congruence $a_p(E)\equiv [2\sqrt{p}] \bmod \ell$. We assume a zero-free region hypothesis for Dedekind zeta functions of number fields.
   }
  
\end{abstract}
\maketitle 
\vspace{-10pt}
\section{Introduction and statement of results}

For a real number $x$, let $[x]$ and $\{x\}$ denote the integer and fractional part of $x$, respectively.
The study of fractional parts of arithmetic sequences has been of great interest and importance. For instance, distribution of $\{\alpha\gamma\}$, where $\alpha$ is a fixed non-zero real number, and $\gamma$ runs over the imaginary parts of zeros of $L$-functions has been studied in various settings, see \cite{Rademacher, Hlawka, Fujii, MR}. Another example would be of lacunary sequences and polynomials, in particular, $\{\alpha n^2\}$ for which we refer the reader to \cite{Rud-Sar-Zah, Rud-Zah-2, Zah, Weyl} and the references therein.

The study of effective equidistribution of $p^\theta$ for $0<\theta <1$ and related sequences such as $\{\alpha p^{\theta}+\beta\}$ is important due to its connection to the infinitude of primes of the form $n^2+1$; this is equivalent to the statement that $\{\sqrt{p}\} <p^{-1/2}$ holds for infinitely many primes $p$. Such sequences have been studied extensively, see \cite{Balog-1,Harman, Baier2004, Baker-Harman-2, Baker-Kolesnik, Murty-Srini} to name a few. A major breakthrough in this direction was achieved by Friedlander and Iwaniec \cite{HI} where they show that the polynomial $x^2+y^4$ captures its primes, followed by Heath-Brown \cite{H-Brown-1} for the primes of the form $x^3+2y^3$. More recently, there has been study of fractional parts of $\sqrt{p}$ with $p$ being of a certain type, namely $p+2\in \mathcal{P}_r$, the set of positive integers with at most $r$ prime factors counted with multiplicity, see for example \cite{Cai,Matomaki-2,Shi-Wu}.

In this paper, we study joint distribution for the fractional parts of prime powers where the primes also satisfy a Chebotarev condition. It is well-known that for $0< \theta <1$, 
$$\#\{p\leq x: \{p^\theta\} <\delta \} \sim \delta\pi(x)$$
with a power saving error term. Here, $\pi(x)$ denotes the number of primes not exceeding $x$.
On the other hand, for a fixed Galois extension $L/\Q$ and a union of conjucgacy classes $C$ in $G:=\mbox{Gal}(L/\Q)$, the celebrated Chebotarev Density Theorem implies that $$\#\{ p \le x: p \text{ unramified in } L, \sigma_p \in C\} \sim \frac{|C|}{|G|}\pi(x)$$ as $x\to\infty$, where $\sigma_p$ is the conjugacy class of the Frobenius automorphism associated with any prime in $L$ lying above $p$. Here the focus is on the distribution of $\{\alpha p^{\theta}\}$ with $\sigma_p\in C$, which we now describe. To simplify exposition, henceforth when we write $\sigma_p \in C$, we assume that $p$ is unramified in $L$, unless mentioned otherwise. 
\begin{thm}\label{Balog+CDT}
     Let $\alpha>0$, $\omega\geq 1$, $0\le\delta_1<\delta_2\le 1$ with $\delta:=\delta_2-\delta_1$, and $\theta \in (0,1)$ be fixed. Let $L/\Q$ be a finite Galois extension with $n_L=[L:\Q]$. 
     For $\theta/2<\Delta \leq 1/2$ fixed, assume that $\zeta_L(s)$ has no zeros in the region $\Re s >1-\Delta$
     and for $\omega\geq 1$, 
     \[ \alpha^{1/4}(\omega n_L/\delta)^{1/2}(\log x)^2 \ll_\theta {x^{\Delta/2-\theta/4}}.
     \]
    Then the following holds uniformly in $\delta, \alpha$ and $\omega$.
   \begin{align*}
    &\nonumber\#\{ x<p\leq 2x : 
    \delta_1
    \leq \{\alpha p^{\theta}\}<\delta_2 \text{ and } \sigma_p \in C \} - \delta \frac{|C|}{|G|}\pi(x) \\
    \nonumber  &\quad \ll_\theta \frac{|C|}{|G|}n_L \log x \left(\frac{(\delta \omega)^{1/2}\alpha^{1/4}}{n_L^{1/2}}{x^{1-\Delta/2 +\theta/4}}
    + \frac{\delta \omega}{\alpha^{1/2}}x^{1-\theta/2} \log x \right)\\
    &\qquad +(\delta n_L \omega)^{1/2}\alpha^{1/4}{x^{\Delta/2 +\theta/4}} 
    \log x + \frac{|C|}{|G|}\frac{\delta x}{\omega \log x}.
    \end{align*}
\end{thm} 
We note that for all our results, setting $\Delta =1/2$ is equivalent to the Generalized Riemann Hypothesis (GRH) for $\zeta_L(s)$. However, we do obtain a power saving over the main term even for weaker zero-free regions, as long as $\Delta>\theta/2$. 

Setting $\alpha =1, \theta=1/2$, and $ \delta_1 =0$ in Theorem \ref{Balog+CDT}, one obtains the following generalization of \cite{Balog-2} that investigates the distribution of fractional parts of $p^{\theta}$ with $p\equiv a \bmod q$.
\begin{corollary} Assume the notations and hypotheses as in Theorem \ref{Balog+CDT}. Then
    \begin{align*}
        &\#\{ x< p\leq 2x : \{\sqrt{p}\}< \delta \text{ and } \sigma_p \in C\}
        -  \delta \pi_C(x,L)\\
        &\ll  \frac{|C|}{|G|}\left( (\delta \omega n_L)^{1/2}x^{{9/8- \Delta/2}}\log x + n_L\delta\omega x^{3/4}(\log x)^{2} + \frac{\delta x}{\omega\log x}\right).
        \end{align*}
\end{corollary}
The above results are of a similar flavour as some other interesting joint distribution theorems such as asymptotics for the number of Charmichael numbers composed of primes satisfying the Chebotarev condition studied in \cite{Carmichael+CDT}. In \cite{Piatetski-Shapiro+CDT}, an asymptotic estimate for the number of Piatetski-Shapiro primes up to $x$ satisfying the Chebotarev condition is derived. 

Using essentially the same ideas as in Theorem \ref{Balog+CDT}, one can prove a similar result, stated below, where the bounds for the fractional part of the prime are themselves a function of the prime. 
\begin{thm}\label{Gen-Landau+CDT}
Let $\alpha, \lambda>0$ and $\theta \in (0,1)$ be fixed. Consider a finite Galois extension $L/\Q$ and let $n_L = [L:\Q]$ and $d_L$ denote the absolute discriminant of $L$. For $\theta/2<\Delta \leq 1/2$ fixed, assume that $\zeta_L(s)$ has no zeros in the region $\Re s >1-\Delta$. Then, for $\omega \ge 1$, 
as $x\rightarrow \infty$,
\begin{align*}
    &\#\{ x<p\leq 2x : 
    \{\alpha p^{\theta}\}<p^{-\lambda} \text{ and } \sigma_p \in C \} - \sum_{\substack{x<p\le 2x \\ \sigma_p \in C}} p^{-\lambda}\\
    &\ll_\theta \frac{|C|}{|G|} \frac{x^{1-\lambda}}{\omega \log x} + \omega{\alpha^{1/2}} \ x^{\theta/2} \log d_L \log^3 x \\
    &\quad + \frac{|C|}{|G|} {\alpha^{1/2}}  \log^3 x \left(\log d_L + n_L \log x\right) \left( \omega^2 x^{{1-\Delta+\theta/2}}+\frac{ \omega x^{1-\lambda -\theta/2}}{\alpha}\right).
\end{align*}
\end{thm}
Setting $\theta=1/2, \alpha=1$ and $\lambda =1/4-\epsilon$ for any $\epsilon>0$, we obtain the following asymptotic result, in the spirit of Landau's prime counting problem where the primes satisfy $\{\sqrt{p}\} <p^{-1/2}$ i.e., $p-1$ is a perfect square.
\begin{corollary}\label{Landau+CDT}
Assume the notations and hypotheses in Theorem \ref{Gen-Landau+CDT}. Then, for any $\epsilon>0$,
\begin{align*}
    & \#\{ x<p\leq 2x : \{\sqrt{p}\}<p^{-1/4+\epsilon} \text{ and } \sigma_p \in C \} - \sum_{\substack{x<p\le 2x\\ \sigma_p \in C}} p^{-1/4+\epsilon}\\
    &\quad \ll \frac{|C|}{|G|} \frac{x^{3/4+\epsilon}}{\omega \log x} + \omega x^{1/4} \log d_L \log^3 x + \frac{|C|}{|G|} \omega^2 x^{{5/4-\Delta}}  \log^3 x \left(\log d_L + n_L \log x\right).
\end{align*}
\end{corollary}
\vspace{1cm}
\subsection{Applications} \label{ExtremalPrimesApp}
Our main application of Theorem \ref{Balog+CDT} is towards the extremal prime counting function for elliptic curves without complex multiplication (CM). Let $E$ denote a non-CM elliptic curve over $\Q$ with conductor $N_E$. For a prime $p$ of good reduction, $E$ reduces to an elliptic curve over the finite field $\F_p$ and we denote by $a_p(E)$ the trace of the Frobenius automorphism acting on the points of $E$ over ${\overline {\F}}_p$. Then, $a_p(E) = p+1-\# E(\F_p)$, and the Hasse bound $|a_p(E) | \leq [2 \sqrt{p}]$ holds. As the name suggests, extremal primes are those for which $[a_p(E)]= \pm [2\sqrt{p}]$, i.e. the primes at the ends of the Sate-Tate distribution. It is not yet known whether there are infinitely many such primes for a single non-CM elliptic curve. We refer the reader to Section \ref{CDTsection} for more details on the background and results on extremal primes in the literature. Using Theorem \ref{Balog+CDT}, we prove the following result on primes satisfying the extremality condition modulo a large enough prime $\ell$. 
\begin{thm}\label{main} Let $E/\Q$ be a non-CM elliptic curve and $L=\Q(E[\ell])$ denote the $\ell$-torsion field of $E$ with $\ell$ being a large prime. For $1/4<\Delta \leq 1/2$ fixed, assume that $\zeta_L(s)$ has no zeros in the region $\Re s >1-\Delta$. Then, for $\ell \ll x^{{2\Delta/9-1/18}}\omega^{-2/9}\log ^{-8/9}x $, $\omega\ge1$,
	\begin{align*}
	 &\# \left\{ x <p \leq 2x: p\nmid N_E,\;
	 a_p(E) \equiv [2 \sqrt{p}]\bmod \ell \right\}- \frac{\pi(x)}{\ell}\\
	 &\quad	\ll_E\frac{x}{\omega \ell \log x} +\omega^{1/2}\ell^{5/4}x^{{9/8-\Delta/2}}\log x + \omega \ell^{7/2} x^{3/4}(\log x)^{2}.
 	\end{align*} 
\end{thm}
The inspiration to study the quantity in Theorem \ref{main} comes from the work of Serre \cite{Serre-Inventiones} and Murty-Murty-Saradha \cite{MMS} towards the Lang-Trotter conjecture, which predicts that for $t\in \Z$ fixed,
\[\# \left\{ p \leq x : p\nmid N_E , \, a_p(E) = t  \right\} \sim C_{E,t} \  x^{1/2}/\log{x}, \quad \text{as } x\to \infty,\] 
where $C_{E,t}$ is a constant depending on $E$ and $t$. In the case of upper bounds towards the Lang-Trotter conjecture, better estimates are known when $a_p(E) =0$. In a similar spirit, we obtain better estimates for the joint distribution in the particular case of $a_p(E) \equiv 0 \mod \ell$. To be precise, the following holds.
\begin{thm}\label{a=0Cor}
Assume the notation and hypotheses as in Theorem \ref{main} and $\ell$ be a large prime with $\ell \ll x^{2\Delta/7-1/14}(\log x)^{-8/7}\omega^{-2/7}$. Then 
\begin{align*}
	   &\# \left\{ x <p \leq 2x: p\nmid N_E, \ 
	a_p(E) \equiv [2 \sqrt{p}]\equiv 0 \bmod \ell\right\} - \frac{\pi(x)}{\ell^2 } \\
	&\quad \ll_E \frac{x}{\ell^2 \omega \log x}+ \frac{\omega^{1/2} x^{9/8-\Delta/2}\log x}{\ell^{1/4}}+\omega \ell^{3/2} x^{3/4}(\log x)^{2}.
\end{align*}
\end{thm}
Moreover, choosing $\omega = 1$ and $\ell= x^{{2\Delta/9-1/18}}\log ^{-8/9}x$ in Theorem \ref{main}, we obtain the following upper bound for the number of extremal primes up to $x$.
\begin{corollary}\label{ExtremalPrimesBound}
  Assume the notation and hypotheses as in Theorem \ref{main}. Then for $x$ large, 
	\[\left\{ x< p \leq 2x : p\nmid N_E, \; 
	a_p(E) = [2 \sqrt{p}] \right\} \ll_E x^{19/18-2\Delta/9}(\log x)^{-1/9}.
	\]
\end{corollary}
\begin{remark}
    Similar results can be obtained for extremal primes at the other end, i.e. when $a_p(E) = - [2 \sqrt{p}]$, using essentially the same arguments as presented here. Moreover, because of the generality in Theorem \ref{Balog+CDT}, one can write versions of Theorem \ref{main} and Corollary \ref{ExtremalPrimesBound} where $\sqrt{p}$ is replaced by $p^\theta$ for $\theta<1/2$.
\end{remark}

The structure of this paper is as follows. We begin with the joint distribution result adapting the ideas of Balog \cite{Balog-1} and Lagarias-Odlyzko \cite{L-O} and prove Theorems \ref{Balog+CDT} and \ref{Gen-Landau+CDT} in Section \ref{section-Balog}. Extremal primes and proofs of Theorem \ref{main} and Theorem \ref{a=0Cor} are discussed in Section \ref{CDTsection}. Lastly, in Section \ref{Proof of F estimate}, we provide details of the results needed in Section \ref{section-Balog}. 
\vspace{3pt}

\noindent{\bf Acknowledgements:} The authors are grateful to Chantal David for suggesting the problem and for many fruitful discussions, and to Terence Tao for bringing Balog's papers to their attention. They thank MSRI and WIN-4 workshop organizers at BIRS for their
support where part of this work was accomplished. The authors would also like to thank Ayla Gafni and Caroline T.-Butterbaugh. The first author is supported by NSF Grant DMS-1854398.
\section{Joint distribution results 
\label{section-Balog}}
\vspace{0.2cm}
In this section, we give proofs of Theorems \ref{Balog+CDT} and \ref{Gen-Landau+CDT}. 

\noindent{\it Remark.} In what follows, we may assume that the parameters $\alpha$ and $\theta$ satisfy $\alpha x^\theta \geq 1$ since for $\alpha x^\theta<1$, $\alpha x^\theta=\{\alpha x^\theta\} \in [\delta_1, \delta_2)$, and therefore the desired quantity can be computed using the Prime Number Theorem.
\subsection{Proof of Theorem \ref{Balog+CDT}}
\begin{proof} 
To start with, we capture the fractional parts in the desired interval as follows:
\begin{eqnarray*}
[\alpha p^\theta - \delta_1] - [\alpha p^\theta - \delta_2] &=&
 \begin{cases} 1 & \mbox{if $\delta_1 \leq \{ \alpha p^\theta \} < \delta_2$;} \\
0 & \mbox{otherwise.} \end{cases}
\end{eqnarray*}
First, we obtain the result when $x_j<p\le x_{j+1}$ for
$x_j:= [x(1+j/B)]+1/2$ for $j=0,\ldots, B$ and $B=[\omega]$. Summing over $j$ then establishes the result for $x<p\le 2x$. 

With $\delta = \delta_2 - \delta_1$, the length of the interval, we set
\begin{eqnarray*}
U_{-} &:=& \frac{\alpha x_j^\theta  }{\delta},  \qquad
U_{+} := \frac{\alpha x_{j+1}^\theta  }{\delta}.
\end{eqnarray*}
Then for $x_j< p \le x_{j+1}$, we have
\[\alpha p^{\theta} \left(1- \frac{1}{U_{-}} \right) - \delta_1 \le 
\alpha p^{\theta} -  \delta_2  \leq \alpha p^{\theta} \left(1- \frac{1}{U_{+}} \right) - \delta_1.\]
Note that we are interested in the sum
\[S:=\sum_{\substack{x_j<p\le x_{j+1}\\ \sigma_p \in C}} [\alpha p^{\theta} - \delta_1 ] - [\alpha p^{\theta} - \delta_2]\]
in order to bound the number of primes $x_j <p\le x_{j+1}$ such that $\sigma_p\in C$ and $\delta_1 \le \{\alpha p^\theta\} < \delta_2.$ This implies
\begin{eqnarray*}
S &\geq& \sum_{\substack{x_j<p\le x_{j+1}\\ \sigma_p \in C}} \left[\alpha p^\theta - \delta_1 \right] - \left[\alpha p^\theta \left(1-\frac{1}{U_{+}} \right) - \delta_1 \right]
\end{eqnarray*}
and 
\begin{eqnarray*}
S &\leq& \sum_{\substack{x_j<p\le x_{j+1}\\ \sigma_p \in C}} \left[\alpha p^\theta - \delta_1 \right] - \left[\alpha p^\theta \left(1-\frac{1}{U_{-}} \right) - \delta_1 \right].
\end{eqnarray*}
Therefore, using 
\[ \frac{\alpha p^\theta}{U_{\pm}}=\delta+\bigO\left(\frac{\delta}{\omega}\right),\]
in order to obtain the claimed asymptotics for $S$, it suffices to prove 
\begin{align}\label{Suff-cond}
&\nonumber\sum_{\substack{x_j<p\le x_{j+1}\\ \sigma_p \in C}} \left[\alpha p^\theta - \delta_1 \right] - \left[\alpha p^\theta \left(1-\frac{1}{U_{\pm}}           \right) - \delta_1 \right]-\frac{\alpha p ^\theta}{U_{\pm}} \\
    &\nonumber \ll \frac{|C|}{|G|}\log x \left( (\delta n_L/\omega)^{1/2}\alpha^{1/4}x^{1-\Delta/2 +\theta/4} + \frac{\delta n_L}{\alpha^{1/2}}x^{1-\theta/2}\log x\right)\\
    &  \quad + (\delta n_L/\omega)^{1/2}\alpha^{1/4}x^{\Delta/2 +\theta/4}\log x.
\end{align}

We may write the summand
\begin{align}\label{sum-of-1's}
\left[\alpha p^\theta - \delta_1 \right] - \left[ \alpha p^\theta \left(1-\frac{1}{U_{\pm}} \right) - \delta_1  \right] & 
= \sum_{m \le\alpha p^\theta-\delta_1}1\  -
 \sum_{m  \leq \alpha p^\theta \left(1-\frac{1}{U_{\pm}} \right)-\delta_1 } 1.
\end{align}
Define the sequence $(a_m)_{m\in \mathbb{N}}$ 
 \begin{equation*}
    a_m:= \begin{cases}
    1 &\text{ if } \frac{1}{3} \alpha x^\theta-\delta_1 < m  \le 3 \alpha x^\theta-\delta_1\\
    0 &\text{ otherwise}.
    \end{cases}
\end{equation*}
Set $$A_{\delta_1}(M) := \sum_{m \leq M-\delta_1} a_m = \sum_{m\geq 1}a_m \ f\left( \frac{m+\delta_1}{M}\right)$$ where 
\begin{equation*}
    f(y)= \begin{cases}
    1 &\text{ if } 0<y<1\\
    1/2 & \text{ if } y=1\\
    0 &\text{ if } y>1.
    \end{cases}
\end{equation*}
Using the inverse Mellin transform, for $\sigma >0$, we write $$f(y) = \frac{1}{2\pi i}\int_{\sigma-i\infty}^{\sigma+ i\infty} \frac{y^{-s}}{s} ~ds.$$

For $y= (m+\delta_1)/M$, this gives
\begin{equation}\label{shifted perron untruncated}
    A_{\delta_1}(M) = \frac{1}{2\pi i}\int_{\sigma-i\infty}^{\sigma+ i\infty} \sum_{m\ge 1} \frac{a_m}{(m+\delta_1)^s}\frac{M^s}{s}~dx.
\end{equation}
Thus, using the definition of the sequence $(a_m)$, we rewrite \eqref{sum-of-1's} as
\begin{align*}
    \left[\alpha p^\theta - \delta_1 \right] - \left[ \alpha p^\theta \left(1-\frac{1}{U_{\pm}} \right) - \delta_1  \right] 
   &=\sum_{m  \le \alpha p^\theta -\delta_1 }a_m\  - \sum_{m  \le (\alpha p^\theta) \left(1-\frac{1}{U_{\pm}} \right)-\delta_1 }a_m\\
 &= A_{\delta_1}(\alpha p^{\theta}) - A_{\delta_1}\left(\alpha p^{\theta}\left(1-\frac{1}{U_{\pm}} \right) \right).
\end{align*}
In order to estimate the integrals $A_{\delta_1}(\alpha p^{\theta})$ and $A_{\delta_1}\left(\alpha p^{\theta}\left(1-\frac{1}{U_{\pm}} \right) \right)$, given by \eqref{shifted perron untruncated}, we make use of the truncated Perron's formula \cite[Lemma 3.12]{Titchmarsh}, and obtain
\begin{align*}
&\nonumber\left[\alpha p^\theta - \delta_1 \right] - \left[ (\alpha p^\theta) \left(1-\frac{1}{U_{\pm}} \right) - \delta_1  \right] 
= \frac{1}{2\pi i} \int_{1/2-i T_1}^{1/2+i T_1} L(s) H(s) \ p^{\theta s} ~ds \\
 & \quad+\bigO\Bigg(\sum_{\frac13 \alpha x^\theta < m+\delta_1 \le 3    \alpha x^\theta}\! \min\left\{1,T_1^{-1}\left|\log\frac{\alpha    p^\theta}{m+\delta_1}\right|^{-1}\right\}\Bigg) \\
 & \quad+\bigO\Bigg(\sum_{\frac13 \alpha x^\theta < m+\delta_1 \le 3    \alpha x^\theta}\!\min\left\{1,T_1^{-1}\left|\log\frac{\alpha p^\theta}{m+\delta_1} \left(1-\frac{1}{U_{\pm}}\right)\right|^{-1} \right\}\Bigg)\\
&\label{PerronError}= \frac{1}{2\pi i} \int_{1/2-i T_1}^{1/2+i T_1} L(s) H(s)  p^{\theta s} ~ds +\bigO\Bigg(\frac{\alpha x^{\theta} }{T_1}\log(\alpha x^{\theta})\Bigg),\\
\end{align*}
where
\[H(s) := \frac{1}{s}\left(1-\left(1-\frac{1}{U_{\pm}}\right)^s\right) \ll \frac{1}{U_{\pm}}\]
and 
\[L(s) := \alpha^s \sum_{\frac13 \alpha x^\theta - \delta_1 < m \le 3 \alpha x^\theta - \delta_1}
\frac{1}{ (m+\delta_1)^{s}}.\]
Note that $T_1$ will be chosen so that $\alpha x^{\theta}\log (\alpha x^{\theta})\ll {\delta} T_1$ so as to obtain a true error term. 
Summing over the primes, we find
\begin{align} 
 &\nonumber\sum_{\substack{x_j<p\le x_{j+1}\\ \sigma_p \in C}} 
 \Bigg(\left[\alpha p^\theta\ - \delta_1 \right] - \left[\alpha p^\theta  \left(1-\frac{1}{U_{\pm}} \right) -\delta_1 \right]\Bigg) \\
&\hspace{.25in}\label{Perron1-error} =  \frac{1}{2\pi i}  \int_{1/2-iT_1}^{1/2+iT_1} L(s) H(s)  F(-\theta s)~ds +  \bigO\left(  \frac{\alpha x^{\theta}}{T_1} \log(\alpha x^{\theta})\pi_C(x_j,L) \right),
\end{align}
where  
\begin{eqnarray*} \label{def-F}
F(s) := \sum_{\substack{x_j<p\le x_{j+1} \\ \sigma_p \in C}} p^{- s} \qquad \text { and } \qquad 
\pi_C(x_j,L) := \sum_{\substack{x_j<p\le x_{j+1}\\ \sigma_p \in C}}1. 
\end{eqnarray*}

 First, we compute the above integral in the smaller range, up to $T_0 :=\alpha x^\theta$. Observe that in the range $|t|\leq T_0$, we have
\begin{eqnarray*} 
H(s) =\frac{1}{U_{\pm}}  + \bigO \left(\frac{|s-1|}{U_{\pm}^2} \right) \label{asymp-H}
\end{eqnarray*} and
\begin{equation*} \label{boundL}
L(s) = \alpha^s  \frac{(3 \alpha x^\theta)^{1-s}  - (\alpha x^\theta/3)^{1-s}}{1-s} + \bigO\left(  x^{-\theta\Re(s)} \right). 
\end{equation*}
Therefore, for the integral in \eqref{Perron1-error} in the range $|t|\leq T_0$, we have
\begin{align}
& \nonumber \frac{1}{2\pi i}  \int_{1/2-iT_0}^{1/2+iT_0} L(s) H(s)F(-\theta s)\,ds \\ \label{integral}
&= \frac{1}{U_{\pm}} \sum_{\substack{x_j < p \le x_{j+1} \\ \sigma_p \in C}} \alpha p^{\theta} \; \frac{1}{2\pi i}  \int_{1/2-iT_0}^{1/2+iT_0} (\alpha p^{\theta})^{s-1} \frac{(3 \alpha x^\theta)^{1-s}  - (\alpha x^\theta/3)^{1-s}}{1-s}~ds \\ 
&\label{ET-integral}
  \quad + \bigO\left( \frac{\delta}{\alpha} x^{-3\theta/2} \int_{-T_0}^{T_0} |F(-\theta/2 - i \theta t)| \; dt \right).
 \end{align}
Using Cauchy-Schwarz inequality, the error term \eqref{ET-integral} is bounded by
 \begin{eqnarray} 
\label{ET-afterCS} &&\ll
\frac{\delta}{\alpha} x^{-3\theta/2} T_0^{1/2} \left( \int_{-T_0}^{T_0} |F(-\theta/2 - i \theta t)|^2~dt  \right)^{1/2}.
\end{eqnarray}
Also, by applying the mean value theorem for Dirichlet polynomials, we have
$$\int_{-T_0}^{T_0} \left\vert F(-\theta/2-i \theta t) \right\vert^2 ~dt  \ll  x^{\theta} 
\pi_C(x_j,L) \;\left( T_0 + x/ \omega \right).$$
Inserting this estimate in \eqref{ET-afterCS}, the error term \eqref{ET-integral} is bounded by
\begin{align}
\nonumber&\ll \delta  T_0^{1/2}(\alpha x^\theta)^{-1}  \pi_C(x_j,L)^{1/2} 
\left(T_0^{1/2}+x^{1/2}/\omega^{1/2}\right)\\
\label{error from integral up to T_0}
&\ll \delta \pi_C(x_j,L)^{1/2} 
\left(1+x^{(1-\theta)/2}(\alpha\omega)^{-1/2}\right),
\end{align}
using $T_0=\alpha x^\theta$. We now compute the integral in \eqref{integral} to obtain the desired main term. The change of variable $w=1-s$ yields
\begin{align*}
 \frac{1}{U_{\pm}} \sum_{\substack{x_j< p \le x_{j+1}\\ \sigma_p \in C}} \alpha p^{\theta}  \frac{1}{2\pi i}  \int_{1/2-iT_0}^{1/2+iT_0} \frac{1}{w} \left(\left(\frac{3  x^\theta}{ p^\theta}\right)^{w}  - \left(\frac{ x^\theta}{3p^\theta}\right)^{w} \right)~dw. 
 \end{align*}
By Perron's formula, for all values $x_j < p \leq x_{j+1}$, 
$$\frac{1}{2\pi i}  \int_{1/2-iT_0}^{1/2+iT_0} \frac{1}{w} \left(\left(\frac{3 x^\theta}{ p^\theta}\right)^{w}  - \left(\frac{ x^\theta}{3p^\theta}\right)^{w} \right)~dw  =1 + \bigO \left( \frac{1}{T_0} \right).$$
Therefore \eqref{integral} becomes
\begin{equation}\label{psi/UT_0}
\sum_{\substack{x_j < p \le x_{j+1}\\ \sigma_p \in C}} \frac{\alpha p^{\theta}}{U_{\pm}} \ + \bigO \left( \frac{1}{T_0 U_{\pm}} \sum_{\substack{x_j < p \le x_{j+1}\\ \sigma_p \in C}}  \alpha p^{\theta} \right).
\end{equation}
Using $T_0 =\alpha x^{\theta}$ and collecting the terms from \eqref{Perron1-error}, \eqref{error from integral up to T_0} and \eqref{psi/UT_0}, we obtain
\begin{align}\label{total error up to T_0}
&\nonumber\sum_{\substack{x_j<p\le x_{j+1}\\ \sigma_p \in C}}  \left(  \left[\alpha p^\theta\ - \delta_1 \right] - \left[\alpha p^\theta  \left(1-\frac{1}{U_{\pm}} \right) -\delta_1 \right] - \frac{\alpha p^{\theta}}{U_{\pm}}\right) \\
&\nonumber \ll  \delta \pi_C(x_j,L)^{1/2} \left(1+\frac{x^{(1-\theta)/2}}{(\alpha\omega)^{1/2}} \right)  +  \frac{\delta}{\alpha x^{\theta}} \ \pi_C(x_j,L) \\
&\quad  + \frac{\alpha x^{\theta}}{T_1} \log(\alpha x^{\theta}) \pi_C(x_j,L) + \frac{1}{U_{\pm}} \int_{T_0}^{T_1}  |L(1/2+it)| |F(-\theta/2-i \theta t)| ~dt.
\end{align} 
In order to estimate the integral in the error term above, we use dyadic division, and bound the following integrals
\[\frac{1}{2\pi i}  \int_{1/2+iT'}^{1/2+i2T'} L(s) H(s)F(-s)~ds \ll \frac1U_{\pm} \int_{T'}^{2T'} |L(1/2+it)| |F(-\theta/2-i \theta t)| ~dt\]
for $T_0 \leq T' \leq T_1/2$ using Cauchy-Schwarz inequality. This is achieved by obtaining the following uniform bound for $T' \leq \tau \leq 2T'$ and $[x_j, x_{j+1}]\subseteq [x,2x] $ established in Proposition \ref{F-estimate-prop} for $K=\Q$,
\begin{align*}
    F(-\theta/2-i \theta \tau) &\ll x^{\theta/2}
    \left(\log d_L+b+\frac{bx\log x}{T'} \right) \\ 
    & \quad + \frac{|C|}{|G|}x^{1-\Delta +\theta/2}\bigg(\log d_L +n_L\log T'\bigg) \left(\frac{\log T'}{\log x} +\frac{x^{\Delta}}{T'} \right);
\end{align*} 
and the mean value estimate below given by Lemma \ref{L-estimate by MV}
$$\int_{T'}^{2T'}|L(1/2+it)|^2~dt \ll \alpha {T'} + \alpha^2 x^{\theta} (\log \alpha x^{\theta}).$$ Note that $\Delta \in (\theta/2, 1/2]$ is fixed.
This gives us
\begin{align*}
 \nonumber &\int_{T'}^{2T'} |L(1/2+it)| |F(-\theta/2 - i \theta t)| ~dt \\
    &\ll \left( \alpha {T'}^2 + \alpha^2 x^{\theta} T'(\log \alpha x^{\theta}) \right)^{1/2}\bigg\{x^{\theta/2}
    \left(\log d_L +b +b \frac{x\log x}{T'} \right) \\
    &\quad +\frac{|C|}{|G|}x^{1-\Delta +\theta/2}\bigg(\log d_L +n_L\log T'\bigg) \left(\frac{\log T'}{\log x} +\frac{x^{\Delta}}{T'} \right)\bigg\}
    & {\vspace{1cm}}\\
    &\ll  x^{\theta/2}  \left(\log d_L + b \right) \left(\alpha {T'}^2 + \alpha^2 T' x^\theta \log(\alpha x^{\theta})\right)^{1/2}\\
    &\quad+  b x^{1+\theta/2} \log x \left(\alpha  + \alpha^2 {T'}^{-1} x^\theta \log(\alpha x^{\theta}) \right)^{1/2}\\
    &\quad+\frac{|C|}{|G|} x^{1-\Delta +\theta/2} \left(\log d_L + n_L\log T'\right)\\
    &\quad \left(\left(\alpha {T'}^2 + \alpha^2 T' x^{\theta} \log(\alpha x^{\theta})\right)^{1/2} \frac{\log T' }{\log x}+  \left(\alpha + \alpha^2 {T'}^{-1} x^{\theta} \log(\alpha x^{\theta}) \right)^{1/2} x^{\Delta}\right).
\end{align*}
Recall that $U_{\pm}\ll \dfrac{\alpha x^{\theta}}{\delta}$, $\alpha x^{\theta} \leq T' < T_1$ and $ \alpha x^{\theta}\log (\alpha x^{\theta})\ll {\delta} T_1$.
Therefore,
\begin{align*}
& \nonumber \frac{1}{U_{\pm}} \int_{T'}^{2T'} |L(1/2+it)| |F(-\theta/2 - i \theta t)| ~dt \\ 
    &\ll \frac{\delta}{(\alpha x^\theta)^{1/2}}\ \bigg\{\frac{|C|}{|G|} x^{1-\Delta} (\log d_L + n_L\log T_1)  \left(T_1 \frac{\log T_1}{\log x}+x^{\Delta} 
    \right) \\
    &\quad+  T_1 \left(\log d_L + b\right) + b x \log x \bigg\}.
\end{align*}
We now set
\begin{align}\label{T_1value?}
T_1 &=\frac{\alpha^{3/4}}{(n_L\delta \omega)^{1/2}\log x}x^{\Delta/2 +3\theta/4}.
\end{align}
Observe that $\log T_1 \ll \log x$ since we have assumed $\displaystyle{\alpha^{1/4}(\omega n_L/ \delta)^{1/2}(\log x)^2 \ll x^{(1-\theta)/4}}.$ Thus, using $\log d_L \ll n_L\log n_L$ (which follows from the discussion in \cite[Section I.3]{Serre-Inventiones}), we conclude
\begin{align*}
    \nonumber \frac{1}{U_{\pm}} \int_{T'}^{2T'} &|L(1/2+it)| |F(-\theta/2 - i \theta t)| ~dt \ll 
    (\delta n_L/\omega)^{1/2}\alpha^{1/4}x^{\Delta/2+\theta/4}\\
    & \quad + \frac{|C|}{|G|}\left((\delta n_L/\omega)^{1/2}\alpha^{1/4}x^{1-\Delta/2 +\theta/4} + \frac{\delta n_L}{\alpha^{1/2}}x^{1-\theta/2}\log x\right).
  \end{align*} 
Given the value of $T_1$, since we use dyadic division of the interval $[T_0, T_1]$, the number of integrals that we need to add is $\bigO(\log x)$. With this, inserting the above estimate in \eqref{total error up to T_0} gives us
\begin{align}
&\nonumber\sum_{\substack{x_j<p\le x_{j+1}\\ \sigma_p \in C}} 
\left(\left[\alpha p^\theta\ - \delta_1 \right] - \left[\alpha p^\theta  \left(1-\frac{1}{U_{\pm}} \right) -\delta_1 \right] - \frac{\alpha p^{\theta}}{U_{\pm}}\right) \\
&\label{line1} \ll  \delta \pi_C(x_j,L)^{1/2} \left(1+\frac{x^{(1-\theta)/2}}{(\alpha\omega)^{1/2}} \right)  +  \frac{\delta}{\alpha x^{\theta}} \ \pi_C(x_j,L) \\ 
& \label{line2} \quad + \frac{|C|}{|G|}\log x \left( (\delta n_L/\omega)^{1/2}\alpha^{1/4}x^{1-\Delta/2 +\theta/4} + \frac{\delta n_L}{\alpha^{1/2}}x^{1-\theta/2}\log x\right)\\
    & \nonumber\label{line3} \quad + (\delta n_L/\omega)^{1/2}\alpha^{1/4}x^{\Delta/2 +\theta/4}\log x.
\end{align}
Note that the error terms in \eqref{line1} can be absorbed into \eqref{line2}. Invoking \eqref{Suff-cond}, this completes the proof of Theorem \ref{Balog+CDT}.
\end{proof}
\vspace{0.3cm}
\subsection{Proof of Theorem \ref{Gen-Landau+CDT}}
Since the proof is similar to that of Theorem \ref{Balog+CDT}, we point out only the main differences below and omit the details.
\begin{proof}
 We proceed as in the proof of Theorem \ref{Balog+CDT} with
\[\delta_1=0, \; \quad \delta_2 = p^{-\lambda}.\]
Note that in this case $\delta=p^{-\lambda} \le x_j^{-\lambda} $ and hence we can eliminate $\delta$ from the error terms. We follow the proof above until \eqref{T_1value?} and choose
\[T_1= \alpha \omega x_j^{\theta+\lambda} \log x .\] 
Following the reasoning after \eqref{T_1value?} in the proof of Theorem \ref{Balog+CDT} with the above value of $T_1$, we obtain the asymptotics when $x_j<p \le x_{j+1}$. Lastly, summing over $j$ gives the desired result claimed in Theorem \ref{Gen-Landau+CDT}. 
\end{proof}
\begin{remark}
   One can also write versions of Theorems \ref{Balog+CDT} and \ref{Gen-Landau+CDT} for a normal extension $L/K$ with Galois group $G=\Gal(L/K)$, where $K/\Q$ is a finite extension and $C$ is a fixed conjugacy class of $G$. Here one would obtain joint distribution results for the primes $\frakp\in K$ unramified in $L$ with the Artin symbol $[\frac{L/K}{\frakp}]=C$ and norm $N_{K/\Q} \frakp \le x $ with conditions on the fractional part $\{\alpha (N_{K/\Q}\frakp)^\theta\}$. For the sake of clarity of exposition, and applications, the proofs are presented with $K=\Q$.
\end{remark}
\vspace{0.3cm}
\section{Applications: extremal primes for non-CM elliptic curves
} \label{CDTsection}

\vspace{0.2cm}
In this section, we provide proofs of Theorems \ref{main} and \ref{a=0Cor}.

 Let $N_E$ denote the conductor of the elliptic curve $E$ without complex multiplication (CM) over $\Q$. Recall that extremal primes for a fixed elliptic curve $E$ are those for which $[a_p(E)]= \pm2\sqrt{p}$. These were first studied by James {\it et al.} \cite{James2016} (see also \cite{JP2017}) who conjectured that, as $x\to \infty$, 
\begin{eqnarray}
	&& \hspace{-1cm} \# \left\{ p \leq x :p\nmid N_E,  a_p(E) = \pm [ 2 \sqrt{p} ] \right\} 
	\sim \begin{cases} \displaystyle \frac{2}{3\pi} \; \frac{x^{3/4}}{\log{x}} & \mbox{for $E$ with CM;} \vspace{2pt}\\   \label{conj-EP-CM}
		\displaystyle \frac{8}{3\pi} \; \frac{x^{1/4}}{\log{x}} & \mbox{for $E$ without CM.}  \end{cases}
\end{eqnarray}
The asymptotics for CM curves were proved in \cite{JP2017} while for the non-CM case, \cite{GJ} confirmed the asymptotics on average. In contrast, the asymptotics \eqref{conj-EP-CM} for a fixed non-CM curve seem to be out of reach with current techniques. 
In \cite{Win4}, along with David, Gafni, and T.-Butterbaugh, the authors established the following upper bounds for a single curve $E/\Q$, under GRH\footnote{The result in \cite{Win4} assumes additional hypotheses which are now known to be true due to the recent breakthrough of Newton-Thorne \cite{newton2020symmetric}.} for $L(s, \mbox{Sym}^n(E)),\ n\geq 0$. 
    \begin{equation*}
		\# \{  x < p \leq  2x : p\nmid N_E,\; a_p(E) = [2\sqrt{p}] \} \ll_E {x^{1/2}}. 
	\end{equation*}
From recent work of Gafni, Thorner and Wong
\cite{gafni2020applications}, it follows that an unconditional upper bound of the order of $x(\log \log x)^2/(\log x)^2$ holds. As an application of Theorem \ref{Balog+CDT}, we study the asymptotics of the following related quantity
\begin{equation}\label{ext_primes_mod_l}
\# \left\{ x <p \leq 2x: p\nmid N_E,\;
	 a_p(E) \equiv [2 \sqrt{p}]\bmod \ell \right\}.
\end{equation}
 We begin by giving some background needed for the proof of Theorem \ref{main}. For a given prime $\ell$, let $E[\ell]$ denote the $\ell$-torsion points subgroup of $E[\bar{\Q}]$. It is known that the Galois representation 
\begin{equation*}
    \rho_{E,\ell}: \mbox{Gal}(\bar{\Q}/\Q) \to \mbox{Aut}_{\F_\ell}(E[\ell]) \cong \mbox{GL}_2(\F_\ell)
\end{equation*} 
is unramified at all primes $p \nmid N_E\ell$. The field $\Q(E[\ell])$, obtained by adjoining the coordinates of all the $\ell$-torsion points of $E$ to $\Q$, is the fixed field in $\bar{\Q}$ of ker$\rho_{E,\ell}$. Serre \cite{Serre-Inventiones} showed that for all but finitely many primes $\ell$, the representation $\rho_{E,\ell}$ is surjective. Thus, using $\rho_{E,\ell}$ we see that for all but finitely many primes $\ell$, the Galois group $G_\ell := \mbox{Gal}(\Q(E[\ell])/\Q)$ is equal to $\mbox{GL}_2(\F_\ell)$. Henceforth when we write $\ell$, we assume that it is a prime large enough so that the surjectivity of $\rho_{E,\ell}$ holds. The characteristic polynomial of $\rho_{E,\ell}(\sigma_p)$ is given by 
$$x^2 -a_p(E) x + p \, (\bmod \ell).$$ 
That is, $a_p(E)$ is the trace of the Frobenius automorphism at $p$. Therefore, for $a \in \F_\ell$, if $C_\ell(a)$ denotes the union of conjugacy classes in $\mbox{GL}_2(\F_\ell)$ of elements of trace $a$ modulo $ \ell$, then 
$$a_p(E) \equiv a \bmod \ell \iff \sigma_p\in C_\ell(a).$$
The structure of conjugacy classes in $\mbox{GL}_2(\F_\ell)$ for an odd prime $\ell$ is well known, see for example \cite[Section 5.2]{Fulton-Harris}. It follows that

\begin{eqnarray} \label{C/G size}
    \frac{|C_\ell(a)|}{|G_\ell|} &=& 
	\begin{cases} \displaystyle  \frac{\ell^2-\ell-1}{(\ell-1)^2(\ell+1)} & \mbox{for $a \neq 0$} \vspace{5pt}\\    
	\displaystyle  \frac{\ell}{(\ell-1)(\ell+1)} & \mbox{for $a = 0$}
	 \end{cases} \\ \nonumber  &=& \frac{1}{\ell} + \bigO \left( \frac{1}{\ell^2} \right).
\end{eqnarray}
Moreover, for $a\neq 0$, $C_\ell(a)$ is a union of $\ell$ conjugacy classes, while $C_\ell(0)$ is a union of $\ell-1$ conjugacy classes. We use the following effective Chebotarev density theorem of Lagarias-Odlyzko \cite{L-O} which is sufficient for our purposes.
\begin{thm*}[\cite{L-O}]
    Let $L/K$ be a finite Galois extension of number fields with Galois group $G$ and $C$ be a conjugacy class in $G$. There exists an effectively computable positive absolute constant $c_1$ such that if $\zeta_L(s)$ satisfies the GRH, then for every $x\geq 2$
    \begin{equation}\label{LO-CDT}
    \pi_C(x,L/K)-\frac{|C|}{|G|}\pi(x) \leq c_1
    \frac{|C|}{|G|}x^{1/2}\log d_Lx^{n_L} + \log d_L.
    \end{equation}
\end{thm*}
\vspace{0.3cm}
\subsection{Proof of Theorem \ref{main}}
\begin{proof}
For each residue $a \in \left\{ 0, 1, \ldots, \ell-1 \right\}$, we have  
$$[2 \sqrt{p}] \equiv a \bmod \ell \iff \left\{\frac{2 \sqrt{p}}{\ell}\right\} \in \left[ \frac{a}{\ell}, \frac{a+1}{\ell} \right)$$
and
$$ a_p(E) \equiv a \bmod \ell \iff \sigma_p \in C_\ell(a) \subseteq G_\ell.$$ 
Therefore, by Theorem \ref{Balog+CDT} with $\theta=1/2,  \alpha =2/\ell,$ and $\delta = 1/\ell$, we have
\begin{align*}
\# &\left\{ p \leq x \;:\; a_p(E) \equiv [2 \sqrt{p}]\equiv a \bmod \ell \right\} \\
&=  \# \left\{ p \leq x : \sigma_p \in C_\ell(a)\;\text{and}\; \left\{\frac{2 \sqrt{p} }{\ell} \right\} \in \left[ \frac{a}{\ell}, \frac{a+1}{\ell} \right) \right\} \\
&=\frac{1}{\ell} \ \pi_{C_\ell(a)}\left(x, \Q(E[\ell])\right)  + \bigO\left(\omega^{1/2} \ell^{1/4} x^{9/8-\Delta/2}\log x +\omega \ell^{5/2}x^{3/4}\log^{2} x + \frac{x}{\omega \ell^2 \log x}\right).
\end{align*}
This gives
\begin{align*}
\# &\left\{ p \leq x: a_p(E) \equiv [2 \sqrt{p}] \bmod \ell \right\} \\
&= \sum_{a \bmod \ell}\# \left\{ p \leq x: a_p(E) \equiv [2 \sqrt{p}] \equiv a \bmod \ell \right\} \\
&= \sum_{a \bmod \ell}  \frac{1}{\ell}\ \pi_{C_\ell(a)}\left(x, \Q(E[\ell])\right) 
 + \bigO\left(\omega^{1/2} \ell^{5/4} x^{9/8-\Delta/2}\log x +\omega \ell^{7/2}x^{3/4}\log^{2} x + \frac{x}{\omega \ell \log x}\right)\\
&= \frac{\pi(x)}{\ell} +\bigO\left(\omega^{1/2} \ell^{5/4} x^{9/8-\Delta/2}\log x + x^{3/4}\ell^{7/2}\log^{5/2} x + \frac{ x}{\omega \ell \log x} 
\right),
\end{align*}
where we have used \eqref{C/G size} and \eqref{LO-CDT} to compute $\pi_{C_\ell(a)}\left(x, \Q(E[\ell])\right)$.
\end{proof}
\vspace{0.3cm}
\subsection{Proof of Theorem \ref{a=0Cor}} 
As in the proof of Theorem \ref{main}, we note that
\begin{align*}
\# &\left\{ p \leq x: a_p(E) \equiv [2 \sqrt{p}] \equiv 0 \bmod \ell \right\}\\
&\quad =\# \left\{ x< p \leq 2x \;:\; \sigma_p \in C_0\;\text{and}\; \left\{\frac{2 \sqrt{p} }{\ell}\right\} \in \left[ 0, \frac{1}{\ell} \right) \right\},
\end{align*}
where $C_0$ denotes the union of conjugacy classes of trace zero in $\mbox{Gal}(L/\Q) =\mbox{GL}_2(\F_\ell)$. 

Before giving the proof, we first fix some notations and record some useful results needed for the proof. For a group $G$ and $C\subset G$, let $\delta_C: G\to \{0,1\}$ denote the class function such that $\delta_C(g)=1$ if and only if $g\in C$. Then,
\[\pi_C(x,L) = \sum_{\substack{ p \text { prime }\\ p \text{ unramified in } L\\{x< p\leq 2x} }} \delta_C(\sigma_p).\] Let
\[\Phi_{C, [\delta_1,\delta_2)}(x,L,\alpha) := 
\sum_{
\substack{
p \text { prime }\\ 
p \text{ unramified in } L\\
{x< p\leq 2x}\\ 
\delta_1\le\{\alpha p^{\theta}\}< \delta_2
}
}
\delta_C(\sigma_p).\]
We now define an analogue of these functions that include contributions from ramified primes as well. Let $D_{\frakp}$ and $I_{\frakp}$ denote the decomposition and inertia subgroups of $G$, respectively at a chosen prime ideal $\frakp$ lying above $p$. Consider $\mbox{Frob}_{\frakp} \in D_{\frakp}/I_{\frakp}$, the Frobenius element at $\frakp$. Then, for each integer $m\geq 1$, we define $$\delta_C(\sigma_p^m) := \frac{1}{|I_{\frakp}|}\sum_{\substack{{g\in D_{\frakp}}\\{ gI_{\frakp}= \mbox{Frob}_{\frakp}^m \in D_{\frakp}/I_{\frakp}}   }} \delta_C(g).$$ Note that $\delta_C(\sigma_p^m)$ is independent of the choice of $\frakp$ and the above definition agrees with the usual definition of $\delta_C(\sigma_p^m)$ for primes $p$ that are unramified in $L$.
Define 
\[\tilde{\pi}_C(x,L) := \sum_{\substack{{p \text{ prime}, \, m\geq 1}\\ {x< p^m\leq 2x}}} \frac{\delta_C(\sigma_p^m)}{m} 
\quad \text{and} \quad \tilde{\Phi}_{C,[\delta_1,\delta_2)}(x,L,\alpha) := \sum_{\substack{{p \text{ prime}, \, m\geq 1}\\ {x< p^m\leq 2x}\\
\delta_1\le\{\alpha p^{m\theta}\}< \delta_2}} \frac{\delta_C(\sigma_p^m)}{m}. \]
With these notations, we state two lemmas from \cite{Zywina} to be used later, and we state them for our case when the base field is $\Q$. 
\begin{lemma}\cite[Lemma 2.7]{Zywina} For any subset $C$ of $G$ stable under conjugation,
    \begin{equation}\label{pi-tilde to pi}
    \tilde{\pi}_C(x, L) = \pi_C(x,L) + \bigO\left( \frac{x^{1/2}}{\log x} + \log d_L\right).
\end{equation}
\end{lemma}
Note that the term $\log d_L$ above makes the error term a little cruder than what appears in \cite{Zywina}, but sufficient for our purposes. The following result follows from Proposition 8 of \cite{Serre-Inventiones}.
\begin{lemma}\cite[Lemma 2.6 (ii)]{Zywina}\label{Zywina-2.6}
    Let $N$ be a normal subgroup of $G$ and let $C$ be a subset of $G$ stable under
conjugation that satisfies $NC\subseteq C$. Then
$$\tilde{\pi}_C(x, L) = \tilde{\pi}_{C'}(x, L^N),$$
where $C'$ is the image of $C$ in $G/N = \mbox{Gal}(L^N/\Q)$.
\end{lemma}
We are now ready to prove Theorem \ref{a=0Cor}.
\begin{proof}
Consider the extension $L/\Q$ with $L= \Q(E[\ell])$. Observe that $C_0$ is stable under multiplication by $H_\ell$,  the subgroup of scalar matrices in $ \mbox{GL}_2(\F_\ell)$. Moreover, it is the inverse image of $C_0'$, the subset of order two elements in $G_\ell' := G_\ell /H_\ell = \mbox{PGL}_2(\F_\ell)$. Applying Lemma \ref{Zywina-2.6}, we obtain $$\tilde{\pi}_{C_0}(x,  L) = \tilde{\pi}_{C_0'}(x, L^{H_\ell}).$$
Therefore, invoking \eqref{pi-tilde to pi}, we have
\begin{align*}
   & \# \left\{ x< p \leq 2x: \sigma_p \in C_0,\; \left\{\frac{2 \sqrt{p} }{\ell}\right\} \in \left[ 0, \frac{1}{\ell} \right) \right\}\\
   &= \Phi_{C_0, [0,\frac{1}{\ell})}\left(x,L, \frac{2}{\ell} \right) \\
   &= \Phi_{C_0',[0,\frac{1}{\ell})}\left(x,L^{H_{\ell}}, \frac{2}{\ell}\right) + \bigO\left(\frac{x^{1/2}}{\log x} + \log d_L\right).
\end{align*} 
Applying Theorem \ref{Balog+CDT} with $\theta=1/2,\ \alpha=\delta=1/\ell$ and the Chebotarev Density Theorem in the form given in \eqref{LO-CDT} to the sub-extension $L^{H_\ell}/\Q$, we conclude
\begin{align*}
     &  \Phi_{C_0',[0,\frac{1}{\ell})}\left(x,L^{H_{\ell}}, \frac{2}{\ell}\right) + \bigO\left(\frac{x^{1/2}}{\log x} + \log d_L\right)\\
    & =  \frac{\pi_{C_0'}(x, L^{H_\ell})}{\ell} +  \bigO\left(\frac{x}{\ell^2 \omega \log x}+ \omega^{1/2}\frac{x^{9/8-\Delta/2}}{\ell^{1/4}}\log x+\omega \ell^{3/2} x^{3/4}(\log x)^{2}\right)
    \\
    & =\frac{\pi(x)}{\ell^2} + \bigO\left(\frac{x}{\ell^2 \omega \log x}+ \omega^{1/2}\frac{x^{9/8-\Delta/2}}{\ell^{1/4}}\log x+\omega \ell^{3/2} x^{3/4}(\log x)^{2}\right),
\end{align*}
for $\ell \ll x^{2\Delta/7-1/14}  \omega^{-2/7} \log ^{-8/7}x$, noting that $\log d_L \ll [L^{H_\ell} :\Q]=|\mbox{PGL}_2(\F_\ell)| < \ell^3 $ and $|C_0'/G_\ell'| = 1/\ell +\bigO(1/\ell^2)$. This completes the proof of Theorem \ref{a=0Cor}. 
\end{proof}

\section{Proofs of intermediate results }\label{Proof of F estimate}
\vspace{0.2cm}
In this section, we give proofs of two auxiliary results used in Section \ref{section-Balog}.
\subsection{Mean value estimation of $L(1/2 +it)$}
\begin{lemma}\label{L-estimate by MV}    
\[\int_{T'}^{2 T'} |L(1/2+it)|^2 ~dt \ll  \alpha T' + {\alpha}^2 x^{\theta} \log(\alpha x^{\theta}) .\]
\end{lemma}
\begin{proof}
We have
	\begin{align*} 
		&\int_{T'}^{2T'} \left|L(1/2+it)\right|^2 ~dt =\alpha \int_{T'}^{2T'} \left|\sum_{\alpha x^\theta/3-\delta_1< m\leq 3\alpha x^\theta-\delta_1}(m+\delta_1)^{-1/2-it}\right|^2 ~dt \\
		&= \alpha \int_{T'}^{2T'} \sum_{\alpha x^\theta/3-\delta_1< m\leq 3\alpha x^\theta-\delta_1}(m+\delta_1)^{-1}+ \sum_{\alpha    x^\theta/3-\delta_1<k\neq m\leq 3\alpha x^\theta-\delta_1}\frac{(k+\delta_1)^{-1/2+it}}{(m+\delta_1)^{1/2+it}}~dt \\
		&= \sum_{\alpha x^\theta/3-\delta_1< m\leq3\alpha x^\theta-\delta_1}\frac{\alpha T'}{(m+\delta_1)}+ \bigO\left( \sum_{ \alpha x^\theta/3-\delta_1<k< m\leq3\alpha x^\theta-\delta_1}\frac{\alpha((m+\delta_1)(k+\delta_1))^{-1/2}}{\log((m+\delta_1)/(k+\delta_1))} \right).
	\end{align*}
	Using $\delta_1 < \alpha x^\theta/3$ and rewriting the sum in the error term, we obtain
	\begin{eqnarray*}
	&&\sum_{\alpha x^\theta/3-\delta_1< m\leq3\alpha x^\theta-\delta_1}\sum_{r=1}^{(m+\delta_1)/9}\frac{\alpha}{(m+\delta_1)\sqrt{1-r/(m+\delta_1}) \log\left(1-r/(m+\delta_1)\right)}\\
		&& \ll \sum_{\alpha x^\theta/3-\delta_1< m\leq3\alpha x^\theta-\delta_1}\sum_{r=1}^{(m+\delta_1)/9}\frac{	\alpha }{r}
		 \ll 	\alpha^2 x^\theta \log (\alpha x^\theta).
	\end{eqnarray*}
   This gives us
	\begin{eqnarray*}
	\int_{T'}^{2T'} |L(1/2+it)|^2 dt &=&\alpha T'\sum_{\alpha x^\theta/3-\delta_1< m\leq 3\alpha x^\theta-\delta_1}(m+\delta_1)^{-1}+ \bigO\left( \alpha^2 x^\theta \log (\alpha x^\theta) \right) \\
	&\ll& \alpha T' + \alpha^2 x^\theta \log(\alpha x^\theta).
	\end{eqnarray*}
	This completes the proof of the lemma.
	\end{proof}
\vspace{0.3cm}
\subsection{Estimation of $F(- \theta / 2 - i \theta t)$} 
For each sub-interval $[x_j, x_{j+1}] \subset[x,2x]$ we prove the following.
\begin{proposition}\label{F-estimate-prop}
Let $\theta \in (0,1)$ be fixed. Consider a finite extension $K$ of $\Q$ and a normal extension $L$ over $K$ with Galois group Gal$(L/K)=G$. Let $C$ be a union of $b$ conjugacy class in $G$, define 
	 $$F(-\theta/2 -i\theta\tau) = \sum\limits_{\substack{x_j < N(\frakp)\leq x_{j+1} \\ \sigma_\frakp \in C}} (N\frakp)^{\theta/2 +i\theta \tau}.$$
Assume that the Dedekind zeta function $\zeta_L(s)$ has no zeros in the strip $\Re s>1-\Delta$ where $\theta/2 <\Delta \leq 1/2$. 
Then 
\begin{align*}
F(-\theta/2 - i \theta \tau)  &\ll \nonumber \frac{|C|}{|G|} x^{1-\Delta+\theta/2}(\log d_L + n_L\log T') \left(\frac{\log T'}{\log x}+ \frac{x^{\Delta}}{T'}\right) \\
& \quad + x^{\theta/2} \left( \log d_L + b n_K + bn_K\frac{ x\log x}{T'}\right)
\end{align*} 
uniformly for $0<T'\leq \tau\leq 2T'$. 
\end{proposition}
The proof follows along the same lines as in \cite{L-O}. The function $F(-\theta/2-i\theta\tau)$ here is similar to that of $\pi_C(x,L/K)$ in \cite{L-O}, the main difference being a shift in the complex variable $s=\sigma +it$ by $\theta/2 +i\theta\tau$. While the shift in the real part by $ \theta/2$ results in a factor of $x^{\theta/2}$ tagging along with the error terms obtained in \cite{L-O}, the shift in the imaginary part is where the saving is obtained. To be precise, we choose a contour that is a box which avoids the real axis, therefore the only poles in the interior are the non-trivial zeros of $L(s,\chi, L/E)$, and a pole at $-\theta/2-i\theta\tau$. In particular, the residue from the pole at $s=1$ that makes up the main term in the proof by Lagarias-Odlyzko \cite{L-O} does not appear here, giving us a power saving under the assumed zero-free region hypothesis, which is equivalent to GRH when $\Delta=1/2$. 

We now provide details of the proof.
\begin{proof}
We first consider the function 
\begin{equation*}
	\Psi_C(-\theta/2 -i\theta\tau):= \sum\limits_{\substack{N\frakp ^m \in (x_j,x_{j+1}]\\ \frakp  \text{:unramified}\\ \left[\frac{L/K}{\frakp }\right]^m = C}} \frac{\log N\frakp }{N\frakp ^{m(-\theta/2-i\theta\tau)}}
\end{equation*} 
for a single conjugacy class $C$. We use partial summation to pass on to the bounds for $F(-\theta/2 -i\theta\tau)$. As in \cite{L-O}, in order to use Hecke $L$-functions, we need to consider the ramified primes as well, (which are later removed). For $\Re s>1$, let
\begin{equation}\label{Z_C-alt-defn}
	Z(s) := -\frac{|C|}{|G|} \sum_{\chi} \bar{\chi}(g) \frac{L'}{L}(s,\chi, L/E),
\end{equation} where $\chi$ runs over the irreducible characters of $H= Gal(L/E)$ and $E$ is the fixed field of $H$, the cyclic subgroup of $G$ generated by a chosen element $g \in C$. Note that
\begin{equation*}
	Z(s) = \sum_{\frakp , m} \Theta(\frakp ^m)\log(N\frakp ) (N\frakp )^{-ms}
\end{equation*}
where for an unramified prime $\frakp  \subseteq \mathcal{O}_K,$ $$\Theta(\frakp ^m)= \begin{cases}
	1 & \text{if } \left[\frac{L/K}{\frakp }\right]^m = C\\
	0 & \text{otherwise}
\end{cases}
$$ and $|\Theta(\frakp ^m)|\leq 1$ if $\frakp $ ramifies in $L$.
We use Perron's formula to estimate $F(-\theta/2 -i\theta\tau)$ by considering partial sums of $Z(s)$, which include the ramified primes as well. Define 
\begin{equation}\label{I_C(x,t) defn}
	I(x_j,T) := I_C(x_j,T,\theta,\tau)= \frac{1}{2\pi i}\int\limits_{\sigma_0-iT}^{\sigma_0+iT} Z(s-\theta/2-i\theta\tau) \frac{{x_{j+1}}^s-{x_{j}}^s}{s}ds
\end{equation} 
where $T = \theta T'/2$ and $\sigma_0=1+\theta/2+1/\log x$. Then, 
\begin{align}\label{error-perron+ram-primes}
	\nonumber|\Psi_C(-\theta/2 -i\theta\tau) - I(x_j,T)|
	& \leq \left| I(x_j,T) - \sum\limits_{\substack{\frakp ,m\\ x_j< N\frakp \leq x_{j+1}}} \frac{\Theta(\frakp ^m)\log(N\frakp )}{(N\frakp )^{m(-\theta/2-i\theta\tau)}}\right| \\
	& \quad + \left| \sum\limits_{\substack{\frakp ,m\\ x_j<N\frakp \leq x_{j+1}\\ \frakp \text{:ramified}}} \frac{\Theta(\frakp ^m)\log(N\frakp )}{(N\frakp )^{m(-\theta/2-i\theta\tau)}}\right|.
\end{align}
The two terms on the right hand side of the above equation are now estimated. Using Lemma 3.1 of \cite{L-O}, we have, 
\begin{align*}
	&\left| I(x_j,T) - \sum\limits_{\substack{\frakp ,m\\ x_j< N\frakp \leq x_{j+1}}} \frac{\Theta(\frakp ^m)\log(N\frakp )}{(N\frakp )^{m(-\theta/2-i\theta\tau)}}\right| \\ 
	& \hspace{1cm} \leq \sum\limits_{\substack{\frakp ,m \\ N\frakp ^m =x_{j+1} \text{ or } N\frakp ^m= x_j}} \left( \frac{\log N\frakp }{(N\frakp ^m)^{-\theta/2 }}+ \frac{\sigma_0}{T}\right) \\
	& \hspace{1.75cm} + \sum\limits_{\substack{\frakp ,m \\ N\frakp ^m \neq x_{j+1} }} \left( \frac{x_{j+1}}{N\frakp ^m}\right)^{\sigma_0} \min \left(1, T^{-1}\left| \log\frac{x_{j+1}}{N\frakp ^m}\right|^{-1} \right)\frac{\log N\frakp }{(N\frakp ^m)^{-\theta/2}} \\
	& \hspace{1.75cm} + \sum\limits_{\substack{\frakp ,m \\ N\frakp ^m \neq x_j  }} \left( \frac{x_j}{N\frakp ^m}\right)^{\sigma_0} \min \left(1, T^{-1}\left| \log\frac{x_j}{N\frakp ^m}\right|^{-1} \right)\frac{\log N\frakp }{(N\frakp ^m)^{-\theta/2}}.
\end{align*}
Following arguments from \cite{L-O} to estimate the terms on the right side of the above inequality, and noting that $x_j\leq 2x$ for each $j=0,\ldots, B$, 
we get 
\begin{equation}\label{Perron error}
	I(x_j,T) - \sum\limits_{\substack{\frakp ,m\\ x_j< N\frakp \leq x_{j+1}}} \frac{\Theta(\frakp ^m)\log(N\frakp )}{(N\frakp )^{m(-\theta/2-i\theta\tau)}} \ll x^{\theta/2}n_K \log x + n_K\frac{\sigma_0}{T}  + n_K\frac{x^{1+\theta/2}\log^2x}{T}.
\end{equation}
Moreover, \begin{align}\label{ram-primes-error}
	\sum\limits_{\substack{\frakp,m\\ x_j< N\frakp \leq x_{j+1}\\ \frakp \text{:ramified}}} \frac{\Theta(\frakp ^m)\log(N\frakp )}{(N\frakp )^{m(-\theta/2-i\theta\tau)}} \ll x^{\theta/2}\log x \log d_L.
\end{align}
Putting \eqref{Perron error} and \eqref{ram-primes-error} together in \eqref{error-perron+ram-primes}, we see that 
\begin{equation}\label{R_1}
\Psi_C(-\theta/2 -i\theta\tau) = I(x_j,T)+\bigO\left( x^{\theta/2}\log x \left( \log d_L + n_K + \frac{n_K x \log x}{T}\right)\right).
\end{equation}
Next, we estimate $I(x_j,T)$. From \eqref{I_C(x,t) defn} and \eqref{Z_C-alt-defn}, we have 
\begin{equation*}
	I(x_j,T) = -\frac{|C|}{|G|} \sum_{\chi} \bar{\chi}(g)\frac{1}{2\pi i} \int\limits_{\sigma_0-iT}^ {\sigma_0 + iT}\frac{L'}{L}\left(s-{\theta}/{2}-i\theta\tau,\chi,L/E\right)\frac{x_{j+1}^s-x_j^s}{s} ~ds,
\end{equation*} where $\chi$ runs through irreducible characters of $H= \langle g\rangle$, $T'\leq\tau \leq 2T'$ is fixed, and $T = \theta T'/2$.
We make the change of variable $s \leftrightarrow s-\frac{\theta}{2} - i\theta\tau$ to rewrite 
$$I(x_j,T) = -\frac{|C|}{|G|} \sum_{\chi} \frac{\bar{\chi}(g)}{2\pi i} \int\limits_{1+\frac{1}{\log x}-iT-i\theta\tau}^ {1+ \frac{1}{\log x}+ iT-i\theta\tau}\frac{L'}{L}\left(s,\chi,L/E\right)\frac{x_{j+1}^{s+\theta/2+i\theta\tau}-x_j^{s+\theta/2+i\theta\tau}}{s+\theta/2+i\theta\tau} ~ds.$$
Abbreviating $\frac{L'}{L}(s,\chi,L/E)$ by $\frac{L'}{L}(s,\chi)$, we evaluate for each character $\chi$ of $H$, the integral
\begin{equation*}\label{I_chi-defn}
	I_{\chi}(x_j,T) := \frac{1}{2\pi i} \int\limits_{1+\frac{1}{\log x}-iT-i\theta\tau}^ {1+ \frac{1}{\log x}+ iT-i\theta\tau}\frac{L'}{L}(s,\chi)\frac{x_j^{s+\theta/2+i\theta\tau}}{s+\theta/2+i\theta\tau} ~ds
\end{equation*} for each $j=0,1,\ldots, B$. We may assume that $T+\theta\tau$ and $ T-\theta\tau$ don't coincide with the imaginary part of a zero of any of the $L(s,\chi)$. To estimate this integral, we move the line of integration and consider the integral over a rectangle and apply Cauchy's theorem. More specifically, for $J:= m+ \frac{1}{2}$ where $m$ is a non-negative integer, let $B_{T,J,\theta}$  be the positively oriented rectangle with vertices at $1+\frac{1}{\log x} -i(T+\theta\tau), \, 1+\frac{1}{\log x} +i(T-\theta\tau),\,  -J-\frac{\theta}{2} +i(T-\theta\tau)$ and $-J-\frac{\theta}{2} -i(T+\theta\tau).$ Observe that this box does not intersect the real-axis, because $T-\theta\tau <0$ for all $\tau \in [T',2T']$. Define 
\begin{equation*}
	I_{\chi}(x_j, T, J) := \frac{x_j^{\theta/2 + i\theta\tau}}{2 \pi i}\int_{B_{T,J,\theta}} \frac{L'}{L}(s,\chi)\frac{x_j^s}{s+\theta/2 +i\theta\tau}~ds.
\end{equation*}
Now we estimate the error term 
\begin{align}\label{R-chi-defn}
R_{\chi}(x_j,T,J) := I_{\chi}(x_j,T,J)-I_{\chi}(x_j,T)
\end{align}
uniformly for each $j=0,\ldots, B$. Here, the error $R_{\chi}(x_j,T,J)$ consists of sum of one vertical integral $V_{\chi}(x_j,T,J)$, and two horizontal integrals $H_{\chi}(x_j,T,J) $ and $H^*_{\chi}(x_j,T)$ which we now estimate, following the line of proof in \cite[Section 6, Lemma 6.2]{L-O}. We deduce
 \begin{align*}
		V_{\chi}(x_j,T,J) &:= \frac{1}{2\pi i} \int\limits_{T}^{-T} \frac{x_j^{-J+it}}{-J+it} \frac{L'}{L}(-J-\theta/2+i(t-\theta\tau),\chi) ~dt 	\\	
		&\ll \frac{x^{-J}}{J}T\left( \log A(\chi) + n_E \log(|T+\theta\tau| + |J+\theta/2| \right);
	\end{align*} 
 \begin{align*}
		H_{\chi}(x_j,T,J) 
		&:=\frac{x_j^{\theta/2+i\theta\tau}}{2\pi i} \int\limits_{-J-\theta/2}^{-1/4} \frac{x_j^{\sigma -iT}}{\sigma+\theta/2-iT} \frac{L'}{L}(\sigma-i(T+\theta\tau), \chi)~d\sigma \\ 
		&\quad - \frac{x_j^{\theta/2+i\theta\tau}}{2\pi i} \int\limits_{-J-\theta/2}^{-1/4}\frac{x_j^{\sigma +iT}}{\sigma+\theta/2+iT} \frac{L'}{L}(\sigma+i(T-\theta\tau), \chi)  ~d\sigma \\ 
		& \ll \frac{x^{-1/4 +\theta/2}}{T\log x} \left( \log A(\chi) + n_E\log|T+\theta\tau| \right);
	\end{align*} 
	and lastly
 \begin{align*}
		H^*_{\chi}(x_j,T) 
		&:= \frac{x_j^{\theta/2+i\theta\tau}}{2\pi i} \int\limits_{-1/4}^{1+1/\log x}  \frac{x_j^{\sigma -iT}}{\sigma+\theta/2-iT} \frac{L'}{L}(\sigma-i(T+\theta\tau), \chi)~d\sigma \\ & \quad- \frac{x_j^{\theta/2+i\theta\tau}}{2\pi i} \int\limits_{-1/4}^{1+1/\log x}\frac{x_j^{\sigma +iT}}{\sigma+\theta/2+iT} \frac{L'}{L}(\sigma+i(T-\theta\tau), \chi)  ~d\sigma \\
		&= \frac{1}{2\pi i} \int_{-1/4}^{1+1/\log x}\frac{ x_j^{\sigma +\theta/2 -iT}}{\sigma+\theta/2-iT} \sum_{\substack{\rho \\ |\gamma-(T+\theta\tau)|\leq 1}} \frac{~d\sigma}{\sigma-i(T+\theta\tau)-\rho}  \\
		&\quad - \frac{1}{2\pi i} \int_{-1/4}^{1+1/\log x}\frac{ x_j^{\sigma +\theta/2 +iT}}{\sigma+\theta/2+iT} \sum_{\substack{\rho \\ |\gamma+(T-\theta\tau)|\leq 1}} \frac{d\sigma}{\sigma+i(T-\theta\tau)-\rho}\\
		& \quad + \bigO\left(\frac{x^{1+\theta/2}}{T\log x}(\log A(\chi) +n_E\log|T+\theta\tau|)\right),
	\end{align*}
where the sum above is taken over the non-real zeros $\rho$ of $L(s,\chi)$ and $A(\chi)= d_EN_{E/\Q}(f_\chi)$, $f_\chi$ being the conductor of $\chi$. The proof of Lemma $6.3$ in \cite{L-O} can be modified to show that
\begin{align*}
 &\int\limits_{-1/4}^{1+1/\log x}\frac{ x_j^{\sigma +\theta/2 -iT}}{\sigma+\theta/2-iT} \sum_{\substack{\rho \\ |\gamma-(T+\theta\tau)|\leq 1}} \frac{1}{\sigma-i(T+\theta\tau)-\rho}~d\sigma \\
 & \hspace{1cm} \ll  \frac{x^{1+1/\log x}}{T\log x} n_\chi(T+\theta\tau)\\
& \hspace{1cm} \ll \frac{x^{1+\theta/2}\log x}{T}(\log A(\chi)+ n_E\log|T+\theta\tau|).
\end{align*}
Here $n_\chi(t)$ denotes the number of zeros $\rho=\beta+i\gamma$ of $L(s,\chi)$ with $0<\beta <1$ and $|\gamma-t|\le 1$. A similar estimate holds for the sum over $\rho$ with $|\gamma+ (T-\theta\tau)|\leq 1$. Therefore, for each $j=0, 1, \ldots, B$, $$H_\chi^*(x_j,T)\ll \frac{x^{1+\theta/2}\log x}{T} \left( \log A(\chi) + n_E\log|T+\theta\tau| \right).$$
Note that the estimate for $H_{\chi}(x_j,T)$ is bounded above by the estimate for $H_\chi^*(x_j,T)$. Therefore, from \eqref{R-chi-defn},
\begin{align}\label{R-chi-estimates}
	\nonumber R_{\chi}(x_j,T,J) &\ll \frac{x^{1+\theta/2}\log x}{T} \left( \log A(\chi) + n_E\log|T+\theta\tau| \right)\\ & + \frac{x^{-J}}{J}T\left( \log A(\chi) + n_E \log(|T-\theta\tau| + |J+\theta/2| \right).
\end{align}
We remark here that one only needs to consider the first term above; the second term goes to zero as $J\to \infty$. Next, by Cauchy's theorem, $I_{\chi}(x_j,T,J)$ is the sum of the residues at the poles of the integrand inside $B_{t,J}$. For our specified contour, the poles occur only at the non-real zeros of $L(s,\chi)$, and at $s=-\theta/2-i\theta\tau$. 
This gives
\begin{equation}\label{Residues}
	I_{\chi}(x_j,T,J) = \sum\limits_{|\gamma +\theta\tau|<T} \frac{x_j^{\rho+\theta/2+i\theta\tau}}{\rho+\theta/2+i\theta\tau} + \frac{L'}{L}(-\theta/2 -i\theta\tau).
\end{equation} 
The term $\frac{L'}{L}(-\theta/2 -i\theta\tau)$ is estimated using the following lemma which is a slightly general version of Lemma 6.2 in \cite{L-O} and can be proved essentially using the same arguments, so we omit the details here. 
\begin{lemma}
	If $s=\sigma+it$ with $\sigma\leq -\theta/2$ and $|s+m| \geq \theta/2$ for all non-negative integers $m$, then $$\frac{L'}{L}(s,\chi) \ll \log A(\chi) +n_E\log(|s|+2).$$
\end{lemma}
Applying these bounds, we get 
\begin{equation}\label{L'/L-pole}
    \frac{L'}{L}(-\theta/2 -i\theta\tau) \ll \log A(\chi) + n_E\log(|\theta/2 +i\theta\tau| +2).
\end{equation}
Using \eqref{L'/L-pole}, \eqref{Residues} and \eqref{R-chi-estimates} in \eqref{R-chi-defn}, we have
\begin{align*}
	I_{\chi}(x_j,T) &= \sum\limits_{|\gamma +\theta\tau|<T} \frac{x_j^{\rho+\theta/2+i\theta\tau}}{\rho+\theta/2+i\theta\tau} \\ & \quad + \bigO\left( \left( \frac{x^{1+\theta/2}\log x}{T} +1 \right) \left( \log A(\chi) + n_E\log|T+\theta\tau| \right) \right).
\end{align*}
We plug this into the definition of $I(x_j,T)$ to obtain
\begin{align}\label{I_x}
\nonumber	&I(x_j,T) - S(x_{j+1}, T) + S(x_j,T)\\
& \nonumber \ll \frac{|C|}{|G|}\left( \frac{x^{1+\theta/2}\log x}{T} +1 \right) \sum_{\chi} \left( \log A(\chi) + n_E\log|T+\theta\tau| \right)\\
	& \ll \frac{|C|}{|G|} \left( \frac{x^{1+\theta/2}\log x}{T} +1 \right)(\log d_L + n_L\log |T+\theta\tau|), 
\end{align} where 
$$S(y,T) := \frac{|C|}{|G|} \sum_{\chi}\bar{\chi}(g)  \sum\limits_{|\gamma +\theta\tau|<T} \frac{y^{\rho+\theta/2+i\theta\tau}}{\rho+\theta/2+i\theta\tau}.$$
Under our assumption of a zero-free region $\Re s>1- \Delta$ where $\theta/2< \Delta \le 1/2$, and with slight modification to the calculation in the proof of Theorem $9.1$ of \cite{L-O}, and summing over $\chi$, we get 
\begin{align*}
	S(x_j,T) \ll \frac{|C|}{|G|}x^{1-\Delta +\theta/2}\log T (\log d_L + n_L\log |T+\theta\tau|)
\end{align*}
Finally, using the above bounds in \eqref{I_x} and recalling \eqref{R_1}, we conclude 
\begin{align*}
	\Psi_C(-\theta/2 -i\theta\tau) 
     &\ll  \frac{|C|}{|G|} x^{1-\Delta +\theta/2} 
    \left(\log d_L + n_L\log |T+\theta\tau|\right) \left(\log T +\frac{x^{\Delta}}{T}\log x\right) \\
	&\quad + x^{\theta/2} \log x\left(\log d_L + n_K + n_K\frac{ x \log x}{T}\right).
\end{align*}
Using partial summation and $\tau\leq 2T'$, and setting $T= \theta T'/2$, one gets the desired bounds for $F$ in the case of a fixed conjugacy class $C$. 

Next, suppose $C = \cup_{m=1}^b C_m$ is a union of $b$ conjugacy classes for some integer $b\geq 1$. Then, using the estimates established for each conjugacy class $C_m$ and summing over $m$, we have 
 \begin{align*}
	F(-\theta/2 -i\theta\tau )  
	& \ll  \frac{\sum_{m=1}^b|C_j|}{|G|} x^{1-\Delta +\theta/2}\bigg(\log d_L + n_L\log T'\bigg) \left(\frac{\log T'}{\log x} + \frac{x^{\Delta}}{T'}\right)\\
	&\quad + x^{\theta/2}  \left( \log d_L + b n_K + bn_K\frac{x \log x}{T'}\right)\\
	& {} \\
	& = \frac{|C|}{|G|} x^{1-\Delta +\theta/2} \bigg(\log d_L + n_L\log T'\bigg) 
	\left(\frac{\log T'}{\log x} + \frac{x^{\Delta}}{T'}\right)\\
    &\quad + x^{\theta/2} \left( \log d_L + bn_K + bn_K\frac{x \log x}{T'} \right),
\end{align*}
noting that the error term in \eqref{ram-primes-error} that contributes $x^{\theta/2}\log d_L$ remains unchanged whether $C$ is a single class or a union of conjugacy classes. This completes the proof of the proposition.
\end{proof}

\bibliographystyle{alpha}

\bibliography{MathematikaSub.bib}

\end{document}

%% file: format.tex
\newtheorem{thm}{Theorem}[section]
\newtheorem*{thm*}{Theorem}

\newtheorem{lemma}[thm]{Lemma}

\newcommand{\beq}{\begin{equation}}
\newcommand{\eeq}{\end{equation}}